\documentclass[11pt]{amsart}
\usepackage{a4wide}
\usepackage{amsmath, amsfonts, amssymb,amscd,graphicx,amsthm}
\usepackage{mathrsfs,url}

\date{\today}

\author[M. Bodirsky]
{Manuel Bodirsky}
\address{Institut f\"{u}r Algebra\\TU Dresden\\01062 Dresden\\Germany}
    \email{Manuel.Bodirsky@tu-dresden.de}
   \urladdr{http://www.math.tu-dresden.de/~bodirsky/}
    \thanks{The first author has received funding from the European Research Council under the European Community's Seventh Framework Programme (FP7/2007-2013 Grant Agreement no. 257039). The third author has been funded through project P27600 of the Austrian Science Fund (FWF). The fourth author has been funded through projects I836-N23 and P27600 of the Austrian Science Fund (FWF)}

\author[D. Evans]
{David Evans}

\address{Department of Mathematics\\  Huxley Building, South Kensington Campus \\ Imperial College London \\ United Kingdom}
\email{david.evans@imperial.ac.uk}
\urladdr{http://wwwf.imperial.ac.uk/~dmevans/}

\author[M. Kompatscher]
{Michael Kompatscher}
    \address{Institut f\"{u}r Computersprachen\\Theory and Logic Group\\Technische Universit\"{a}t Wien\\Favoritenstrasse 9/E1852\\A-1040 Wien\\
Austria}
\email{michael@logic.at}
\urladdr{https://www.logic.at/staff/kompatscher/}

\author[M. Pinsker]
{Michael Pinsker}
	\address{Department of Algebra, MFF UK, Sokolovska 83, 186 00 Praha 8, Czech Republic}    
    \email{marula@gmx.at}
    \urladdr{http://dmg.tuwien.ac.at/pinsker/}

\title[A counterexample to reconstruction from monoids]{A counterexample to the reconstruction\\ of $\omega$-categorical structures\\ from their endomorphism monoid}

\theoremstyle{plain}

    \newtheorem{thm}{Theorem}[section]
    \newtheorem{theorem}[thm]{Theorem}

    \newtheorem{lem}[thm]{Lemma}
    \newtheorem{lemma}[thm]{Lemma}

    \newtheorem{prop}[thm]{Proposition}
    \newtheorem{proposition}[thm]{Proposition}

    \newtheorem{cor}[thm]{Corollary}

    \newtheorem{quest}[thm]{Question}

\theoremstyle{definition}

    \newtheorem{nota}[thm]{Notation}

\newcommand{\Ac}{\mathcal A}

\newcommand{\cC}{\mathcal C}

\DeclareMathOperator{\Sym}{Sym}

\newcommand{\rest}{|}

\newcommand{\cA}{\ensuremath{\mathcal{A}}}
\newcommand{\cB}{\ensuremath{\mathcal{B}}}
\newcommand{\cM}{\ensuremath{\mathcal{M}}}

\newcommand{\ignore}[1]{}

\newcommand{\To}{\rightarrow}

\newcommand{\mult}{\times}

\DeclareMathOperator{\Aut}{Aut}
\DeclareMathOperator{\End}{End}
\DeclareMathOperator{\Pol}{Pol}

\newcommand{\A}{{\bf A}}

\newcommand{\G}{{\bf G}}

\renewcommand{\H}{\mathbf H}
\newcommand{\F}{\mathbf F}
\newcommand{\R}{\mathbf R}
\newcommand{\M}{\mathbf M}

\begin{document}
\begin{abstract}
We present an example of two countable $\omega$-categorical structures, one of which has a finite relational language, whose endomorphism monoids are isomorphic as abstract monoids, but not as topological monoids -- in other words, no isomorphism between these monoids is a homeomorphism. For the same two structures, the automorphism groups and polymorphism clones are isomorphic, but not topologically isomorphic. In particular, there exists a countable $\omega$-categorical structure in a finite relational language which can neither be reconstructed up to first-order bi-interpretations from its automorphism group, nor up to existential positive bi-interpretations from its endomorphism monoid, nor up to primitive positive bi-interpretations from its polymorphism clone.
\end{abstract}

\maketitle

\section{Introduction and the Result}
\label{sect:intro}
How much information about a structure $\mathcal A$ is 
coded into its automorphisms group $\Aut(\mathcal A)$?
Classical model theory provides 
a strong form of reconstruction of ${\mathcal A}$ 
from $\Aut({\mathcal A})$ 
when $\Aut({\mathcal A})$ is big
in the sense that it has for all $k \geq 1$
only finitely many orbits in the componentwise
action on $k$-tuples of elements of ${\mathcal A}$;
 such permutation groups are called 
\emph{oligomorphic}. By the theorem of
Ryll-Nardzewski, the automorphism group 
of a countable structure ${\mathcal A}$ is oligomorphic if and
only if ${\mathcal A}$ is \emph{$\omega$-categorical}, that is, all countable 
models of the first-order theory of ${\mathcal A}$ are isomorphic. 
The classical reconstruction result for an $\omega$-categorical structure ${\mathcal A}$ states that 
when $\Aut({\mathcal A})$ is 
equal to $\Aut({\mathcal B})$ as a permutation group for some structure ${\mathcal B}$, then ${\mathcal B}$ has a first-order definition in ${\mathcal A}$, and vice versa: the two structures are \emph{first-order
interdefinable}. The assumption that
${\mathcal A}$ is $\omega$-categorical is in some sense best possible for
this type of reconstruction: it can be seen that
when ${\mathcal A}$ has a countable signature, 
then the above reconstruction statement holds if and only if ${\mathcal A}$ is $\omega$-categorical. 

The situation is more complicated
when we only know that 
$\Aut({\mathcal A})$ and $\Aut({\mathcal B})$ are isomorphic as groups. To approach this
question, it is essential to first examine 
$\Aut({\mathcal A})$ and $\Aut({\mathcal B})$
as topological groups, equipped with
the topology of pointwise convergence. 
With this topology, automorphism groups
of countable structures are precisely the closed
subgroups of the full symmetric group $\Sym(\omega)$ on $\omega$.  
A result due to Coquand (see~\cite{AhlbrandtZiegler}) says that when $\Aut({\mathcal A})$
and $\Aut({\mathcal B})$ are isomorphic as 
topological groups (that is, via an isomorphism
that is also a homeomorphism),
then ${\mathcal A}$ and ${\mathcal B}$ are
first-order bi-interpretable. We do not require the notions of interpretability and bi-interpretability here, and refer to~\cite{AhlbrandtZiegler} for details, but mention
that these notions are central in model theory since most model-theoretic concepts
are stable under bi-interpretability. Hence, we focus on  
a subproblem:
is it true that when $\Aut({\mathcal A})$ and
$\Aut({\mathcal B})$ are isomorphic as groups, then they are also isomorphic as topological groups?

Rather surprisingly, isomorphisms between automorphism groups of countable structures are typically homeomorphisms. 
And in fact, it is consistent with ZF + DC that
\emph{all} homomorphisms between 
closed subgroups of $\Sym(\omega)$ are continuous, and that all isomorphisms
between closed subgroups of $\Sym(\omega)$
are homeomorphisms; see the end of Section~\ref{sect:profinite} for more explanation. 
Using the existence of non-principal ultrafilters on $\omega$, it is relatively easy to show that there are oligomorphic permutation groups
with non-continuous homomorphisms
to ${\mathbb Z}_2$. 
But it was open for a while whether for countable $\omega$-categorical structures $\mathcal A$ and $\mathcal B$ 
the existence of an isomorphism between
$\Aut({\mathcal A})$ and $\Aut({\mathcal B})$ 
implies the existence of an isomorphism
which is additionally a homeomorphism.
This problem was solved by the second author and Hewitt~\cite{EvansHewitt}, by giving two structures $\mathcal A$ and $\mathcal B$ for which the answer was negative.

Natural objects that carry more information
about a structure 
${\mathcal A}$ than $\Aut({\mathcal A})$ are its
\emph{endomorphism monoid} $\End({\mathcal A})$, which consists
of the set of homomorphisms from ${\mathcal A}$ to ${\mathcal A}$,
or, even more generally,
its \emph{polymorphism clone} $\Pol({\mathcal A})$, which consists of
the set of all homomorphisms from ${\mathcal A}^k$ to ${\mathcal A}$,
for all $k \geq 1$. We are going to show the following theorem related to results of Lascar~\cite{Lascar-demigroupe}; see also the discussion in Section~\ref{sect:open}.

\begin{theorem}\label{thm:main}
There are countable $\omega$-categorical structures
$\mathcal A$, $\mathcal B$ such that $\End(\mathcal A)$
and $\End(\mathcal B)$ are isomorphic, but not topologically isomorphic.
\end{theorem}

In fact, the two endomorphism monoids of the structures $\mathcal A$
and $\mathcal B$ will be the closures in $\omega^\omega$ of the two
automorphism groups which are isomorphic, but not topologically
isomorphic, presented in~\cite{EvansHewitt}.
Ironically, it is its non-continuity which makes the extension of the
isomorphism between those groups to their closures non-trivial, giving
rise to the present work.

It has been asked in~\cite{Reconstruction} whether there are
$\omega$-categorical
structures 
whose polymorphism clones are isomorphic,
but not topologically. Theorem~\ref{thm:main} immediately implies a
positive answer to this question: any two structures whose
polymorphism clones consist essentially (that is, up to adding of
dummy variables)  of the functions in $\End(\mathcal A)$
and $\End(\mathcal B)$, respectively, are examples.

\begin{cor}\label{cor:main}
There are countable $\omega$-categorical structures
$\mathcal A$, $\mathcal B$ such that $\Pol(\mathcal A)$
and $\Pol(\mathcal B)$ are isomorphic, but not topologically isomorphic.
\end{cor}

The construction in \cite{EvansHewitt}
is based on a representation of 
profinite groups as quotients of
oligomorphic groups, due to Hrushovski, and on a non-reconstruction result for profinite groups which uses the
axiom of choice. The non-reconstruction lifts to the oligomorphic groups representing the profinite groups. 

In the present paper we show that it lifts further to the closures of the oligomorphic groups.  
The method of embedding profinite groups into quotients of oligomorphic structures is quite powerful and might 
be useful in different contexts as well. 

The structures constructed in our proof of Theorem~\ref{thm:main} have an infinite relational language. We use a well-known construction due to Hrushovski to encode countable $\omega$-categorical structures into structures with a finite relational language, and show that this encoding is compatible with our examples, roughly because the encoding preserves model-completeness. That way, we obtain the following main theorem of the present article.

\begin{thm}\label{thm:finitelanguage}
There exists a countable $\omega$-categorical structure
$\mathcal A$ in a finite relational language such that none of $\Aut({\mathcal A})$, $\End({\mathcal A})$, and $\Pol({\mathcal A})$ have \emph{reconstruction} (cf.~\cite{Reconstruction}): that is, there exists a countable $\omega$-categorical structure
$\mathcal B$ such that $\Aut({\mathcal A})$ 
 and $\Aut({\mathcal B})$, $\End({\mathcal A})$ and $\End({\mathcal B})$, and $\Pol({\mathcal A})$ and $\Pol({\mathcal B})$ are isomorphic, but not topologically isomorphic.
\end{thm}

\section{Preliminaries}
\label{sect:prelims}

A \emph{topological monoid} $\M = (M, \cdot,1)$ is a monoid together with a topology on $M$ such that the multiplication $\cdot\colon M^2 \to M$ is a continuous function. A \emph{topological group} $\G = (G, \cdot, ^{-1},1)$ is a group such that $(G, \cdot,1)$ is a topological monoid and also $^{-1}\colon G \to G$ is continuous.

Every permutation group $\Sigma$ on a set $X$ (and, likewise, every transformation monoid $\Lambda$ on $X$) gives rise to a topological group (a topological monoid) as follows. 
We equip $X$ with the discrete topology,
and $X^X$ with the product topology. 
Then composition of transformations in $\Lambda$,
and composition and taking the inverse of permutations in $\Sigma$ are continuous with respect to the subspace topology inherited from $X^X$.  
We write $\Sym(X)$ for the 
set of all permutations of the set $X$. 
If 
$\Sigma$ is a permutation group on a set $X$ and $A \subset X$, the (pointwise) stabilizer of $A$ in $\Sigma$ is denoted by 
$\Sigma_{(A)}$.

A transformation monoid is closed in $X^X$ if and only
if it is the endomorphism monoid of a relational structure.
Likewise, a permutation group is closed in $\Sym(X)$ if and only if 
it is the automorphism group of a structure with domain
$X$.
The topological groups that arise in this way as automorphism groups of countable structures 
are precisely those Polish groups that have a compatible
left-invariant ultrametric~\cite{BeckerKechris}. 

For a subgroup $\H$ of $\G$ we write $\H \leq \G$,
and we write $g \H := \{g h : h \in H\}$ for the (left-) coset of $\H$ in $\G$ containing $g$. We denote by $\G / \H$ the set 
of all cosets of $\H$ in $\G$. 
If $\H$ is a normal subgroup of $\G$ then 
$\G / \H$ carries a natural group structure 
which is a topological group with respect to the quotient topology. We write $\G \cong \H$ if $\G$ and $\H$
are isomorphic as groups, and $\G \cong_T \H$
if $\G$ and $\H$ are \emph{topologically isomorphic},
that is, there exists an isomorphism which is also a
 homeomorphism. 
When forming direct products
$\G \times \H$
of topological groups $\G$ and $\H$, then
the group $\G \times \H$ is equipped with the product topology of $\G$ and $\H$.

For background on profinite
groups, we refer to the text book of Ribes and Zalesskii~\cite{RibesZalesskii}. 

\emph{Function clones} are the multivariate generalisation 
of transformation monoids. For a fixed set $X$, the
largest function clone on $X$ is the set $O(X) := \bigcup_{k \geq 1} X^{X^k}$, and a function clone (on $X$)
is a subset of $O(X)$ (called the \emph{operations}) that contains all the projection maps 
and that is closed under composition. 
Each set $X^{X^k}$ is equipped with the product topology (again, $X$ is taken to be discrete), 
and $O(X)$ then carries the sum topology. With respect to this topology,
composition of operations is continuous, and the clones that are closed
subsets of $O(X)$ are precisely the polymorphism clones
of structures with domain $X$. 

A \emph{clone homomorphism} from a function clone $\Gamma$ to a function clone $\Delta$ 
 is a map $\xi$ from the operations
of $\Gamma$ to the operations of $\Delta$ 
such that for all $f,g_1,\dots,g_n \in \Gamma$
we have $\xi(f(g_1,\dots,g_n)) = \xi(f)(\xi(g_1),\dots,\xi(g_n))$. 
A clone isomorphism is a bijective clone homomorphism.
We refer the reader to~\cite{Reconstruction} for a more thorough treatment of function clones
and topological clones.

\section{The Proof}
\label{sect:proof}

\subsection{Overview}
The idea is to obtain the results in the following steps.
\begin{itemize}
\item[(1)] There exist separable profinite groups $\G$ and $\G'$ which are abstractly but not topologically isomorphic: $\G\cong\G'$ but $\G\ncong_T\G'$.
\item[(2)] There is a oligomorphic permutation group $\Phi$ on a countable set such that for every separable profinite group $\R$ there exists a closed permutation group $\Sigma_\R \geq \Phi$ such that $\R\cong_T \Sigma_\R / \Phi$.
Furthermore $\Phi$ can be characterized in the topological group structure of $\Sigma_\R$ as the intersection of the open normal subgroups of finite index.
\end{itemize}
It would then be natural to continue by the following steps. However, we do not know whether (3) is true, so the argument will proceed in a less direct way, but still following the outline below.
\begin{itemize}
\item[(3)] For the separable profinite groups $\G$ and $\G'$ from (1), the permutation groups $\Sigma_\G$ and $\Sigma_{\G'}$ are isomorphic.
\item[(4)] $\Sigma_\G$ and $\Sigma_{\G'}$ cannot be topologically isomorphic, since by (2) any topological isomorphism would have to send $\Phi$ onto itself, and so $\Sigma_{\G}/ \Phi$ and $\Sigma_{\G'}/ \Phi$ would be topologically isomorphic, contradicting (1).
\item[(5)] The isomorphism between the permutation groups $\Sigma_\G$ and $\Sigma_{\G'}$ extends to their topological closures $\overline{\Sigma_\G}$ and $\overline{\Sigma_{\G'}}$ in $\omega^\omega$. However, the closed monoids $\overline{\Sigma_\G}$ and $\overline{\Sigma_{\G'}}$ are not topologically isomorphic: otherwise we would obtain a topological isomorphism between $\Sigma_\G$ and $\Sigma_{\G'}$ by restricting any topological isomorphism between $\overline{\Sigma_\G}$ and $\overline{\Sigma_{\G'}}$, contradicting~(4).
\item[(6)] The closed oligomorphic function clones containing precisely the essentially unary functions obtained from $\overline{\Sigma_\G}$ and $\overline{\Sigma_{\G'}}$ are isomorphic by extending the isomorphism between $\overline{\Sigma_\G}$ and $\overline{\Sigma_{\G'}}$ naturally. However, they are not topologically isomorphic as otherwise $\overline{\Sigma_\G}$ and $\overline{\Sigma_{\G'}}$ would be topologically isomorphic as well by restricting any topological isomorphism between the functions clones to their unary sort.
\item[(7)] $\Sigma_\G$ can be encoded in a structure in a finite language such that the above arguments still work.
\end{itemize}

We remark that the steps (1)-(3) have already been discussed in \cite{EvansHewitt}, but we are going to recapitulate them for the convenience of the reader and to build on the construction in the further steps. The profinite group $\G$ in (1) has been known for a long time~\cite{Witt}. Its properties were used in~\cite{EvansHewitt} to construct the profinite group $\G'$ that is isomorphic, but not topologically isomorphic to it. The proof of step (2) is due to an idea of Cherlin and Hrushovski, and (7) to another idea of Hrushovski.

The biggest technical challenge is step (3), and similarly, step (5). It is worth noting that we do not know whether (3) and (5) are true in general; our proof depends on the particular structure of the group $\G$ from~(1). In fact, our proof will deviate from the above presentation in that we will not directly work with $\G$ but with a factor thereof. We find it, however, useful to have the above schema in mind since it does reflect the general proof idea.

\subsection{Profinite groups}
\label{sect:profinite}

In this section we are going to discuss the profinite group $\G$ that will be the basis of our counterexample. We say a subgroup $\F' \leq \G$ is a \textit{complement} of a normal subgroup $\F$ of $\G$ iff $\G = \F \cdot \F'$ and $\F \cap \F'$ is the identity subgroup.

\begin{proposition}\label{prop:profinite_1}
There exists a separable profinite group $\G$ with the following properties:
\begin{itemize}
\item $\G$ has a non-trivial, finite central subgroup $\F$ with a dense complement $\F'$ in $\G$;
\item any complement of any finite central subgroup of $\G$ is dense in $\G$.
\end{itemize}
\end{proposition}

The construction of this profinite group can be found in~\cite[Theorem~4.1]{EvansHewitt},  where it is also used to answer a question about relative categoricity.  We remark that the same group had already been constructed in~ \cite{Witt} in a different context, namely to provide an example of a compact separable group with a non-compact commutator subgroup.

\begin{lemma}\label{lem:profinite_2}
Let $\G, \F$ and $\F'$ be as in Proposition~\ref{prop:profinite_1}. Then:
\begin{itemize}
\item $\G/\F$ is a profinite group which is isomorphic, but not topologically isomorphic to $\F'$;
\item $\G$ and $\G/\F\mult \F$ are isomorphic as groups, but are not topologically isomorphic.
\end{itemize}
\end{lemma}
\begin{proof}
Since $\F$ is central we have that $\G = \F' \cdot \F$. Since moreover $\F' \cap \F$ is the identity subgroup, every $g \in \G$ has a unique representation $g = f' f$, where $f' \in \F'$ and $f \in \F$. Hence every coset $g \F$ contains exactly one representative from $\F'$. So the restriction of the quotient homomorphism $\G \to \G/\F$ to $\F'$ is bijective and thus an isomorphism. Since $\F$ is closed, $\G/\F$ is a profinite group; in particular it is compact. By Proposition \ref{prop:profinite_1} $\F'$ is not closed in $\G$ and therefore not compact. So $\G/\F$ and $\F'$ cannot be topologically isomorphic.

Since $\F$ is central in $\G$, we have that $\F' \mult \F$ is isomorphic to $\G$, and so is $\G/\F\mult \F$ by the above. However, no isomorphism from $\G/\F \mult \F$ to $\G$ can be a topological one. Otherwise, the image of $\F$ (viewed as a subgroup of $\G/\F \mult \F$ in the natural embedding) would be central in $\G$ and so the image of $\G/\F$ would have to be a proper dense subgroup of $\G$, by Proposition~\ref{prop:profinite_1}. Therefore it would not be closed, contradicting compactness.
\end{proof}

\begin{nota}\label{nota:profinite_3}
From now on, we fix groups $\G$, $\F$, and $\F'$ as in Proposition~\ref{prop:profinite_1}. We moreover denote the isomorphism from $\G/\F$ onto $\F'$ which sends every class $g\F$ to the unique element in $g\F\cap \F'$ by $\kappa$.
\end{nota}

We remark that the axiom of choice was used to show the existence of the pair of subgroups $\F, \F'$ in $\G$ in Proposition~\ref{prop:profinite_1}. This seems unavoidable: it is well-known that every Baire measurable homomorphism between Polish groups is continuous (see e.g.~\cite{Kechris}). Further the statement that every set is Baire measurable is consistent with ZF+DC (\cite{Shelah84}). Thus the existence of two separable profinite groups (respectively two closed oligomorphic groups) that are isomorphic, but not topologically isomorphic, cannot be proven in ZF+DC (see the discussion in~\cite{Topo-Birk}). The  insufficiency of  ZF+DC to construct a non-continuous homomorphism between Polish groups  was  already  observed in \cite{Lascar91}.

\subsection{Encoding profinite groups as factors of oligomorphic groups}

The next step is to describe a given separable profinite group as a factor of two oligomorphic permutation groups. Our argument is a generalization of an argument of Cherlin and Hrushovski, which can be used to show that there are oligomorphic groups without the small index property \cite{Lascar82}. A similar construction is also used in \cite{BPP-projective-homomorphisms} to show that there is an oligomorphic clone on a countable set with a discontinuous homomorphism onto the projection clone. The result also appears in \cite{EvansHewitt}.

\begin{prop}\label{prop:encoding_1}
There is a closed oligomorphic permutation group $\Phi$ on a countable set $X$ such that for any separable profinite group $\R$ there exists a closed permutation group $\Sigma_\R$ such that $\Phi \leq \Sigma_\R \leq \Sym(X)$ and:
\begin{itemize}
\item $\Phi$ is a closed normal subgroup of $\Sigma_\R$,
\item $\Phi$ is the intersection of the open subgroups of $\Sigma_\R$ of finite index,
\item $\R\cong_T \Sigma_\R / \Phi$.
\end{itemize}
\end{prop}

\begin{proof}

We first prove the proposition for the special case $\R = \prod_{n \geq 1} \Sym(n)$. Let $L$ be the language containing an $n$-ary relation symbol $P_i^n$ for all integers $1 \leq i \leq n$. Then we consider the class of all finite $L$-structures such that
\begin{itemize}
\item for all $n \geq i \geq 1$: $P_i^n(\bar x)$ implies that the entries of $\bar x$ are distinct;
\item for all $n\geq 1$: $P_1^n$, \ldots, $P_n^n$ form a partition of the $n$-tuples with distinct entries.
\end{itemize}

It is easy to verify that this class is a Fra\"iss\'e-class. Thus there is a unique countable homogeneous structure $\cA^* = (A,(P_i^n)_{n \geq i \geq 1})$ whose \emph{age}, i.e., its set of finite induced substructures up to isomorphism, is equal to this class. Since the number of relations of any fixed arity in  $\cA^*$ is finite, $\cA^*$ is $\omega$-categorical. We set $\Phi$ to be the automorphism group of $\cA^*$.

For every $n$, let $E^n(\bar x,\bar y) $ be the $2n$-ary relation on $A$ that holds if and only if $\bar x$ and $\bar y$ are members of the same partition class $P_i^n$. By definition, the relation $E^n$ forms an equivalence relation on the $n$-tuples with distinct entries that has the sets $P_i^n$ as equivalence classes. We set $\Sigma_\R$ to be the automorphism group of $(A,(E^n)_{n \geq 1})$. Clearly every $E^n$ is definable in $\cA^*$, so $\Phi \leq \Sigma_\R$. By verifying that $(A,(E^n)_{n \geq 1})$ has the extension property, one can easily see that it is a homogeneous structure.

Every function in $\Sigma_\R$ induces a permutation on the set $X_n := \{P_i^n : 1\leq i \leq n\}$, for every $n \geq 1$. The action of $\Sigma_\R$ on the disjoint union of $X_n$ gives us a homomorphism $\mu_\R \colon \Sigma_\R \to \R$. The homogeneity of $(A,(E^n)_{n \geq 1})$ guarantees that every permutation on a finite subset of $\bigcup_{n \geq 1} X_n$ (respecting the arities $n$) is induced by an element of $\Sigma_\R$. This fact, together with a standard back-and-forth-argument, implies that we can obtain every permutation on the full union $\bigcup_{n \geq 1} X_n$ as the action of an element of $\Sigma_\R$. In other words, $\mu_{\R}$ is surjective. Every stabilizer in $\Sigma_\R$ of a finite subset of $\bigcup_{n \geq 1} X_n$ is an open subgroup, hence $\mu_\R$ is continuous and open. The kernel of $\mu_\R$ is $\Phi$, so we have $\Sigma_\R / \Phi \cong_T \R$. 

Finally, we want to prove that $\Phi$ is the intersection of the open subgroups of $\Sigma_\R$ of finite index. It is clear that $\Phi$ contains this intersection, since $\Phi$ is the intersection of the preimages of all the stabilizers of $X_n$, $n \geq 1$, under the action of $\mu_\R$. It remains to show that $\Phi$ has no proper open subgroup of finite index.

Suppose that $\Phi$ has a proper open subgroup $\Lambda \leq \Phi$ of finite index. Since $\Lambda$ is open, there is a finite tuple $\bar y$ of distinct elements in $A$ such that its stabilizer $\Phi_{(\bar y)}$ lies entirely in $\Lambda$. We will obtain a contradiction by studying the actions of $\Phi$ and $\Lambda$ on $\bar y$. Let $O_{\Phi}(\bar y) := \{ g(\bar y) : g \in \Phi \}$ and  $O_{\Lambda}(\bar y) := \{ g(\bar y) : g \in \Lambda \}$ be the orbits of $\bar y$ under these actions. Now $O_{\Phi}(\bar y)$ can be partitioned into subsets of the form $g O_{\Lambda}(\bar y)$, where $g \in \Phi$. This partition is clearly preserved under the action of $\Phi$. For all $g \in \Phi$ the following holds:
\[ g( \bar y) \in O_{\Lambda}(\bar y) \Leftrightarrow \exists h \in \Lambda ( g(\bar y) = h(\bar y) ) \Leftrightarrow \exists h \in \Lambda ( h^{-1} \circ g \in \Phi_{(\bar y)} ) \Leftrightarrow g \in \Lambda. \]
Thus the index $|\Lambda : \Phi|$ coincides with the number of partition classes $g O_{\Lambda}(\bar y)$ in $O_{\Phi}(\bar y)$. Since this index is greater than 1, there exists $\bar b \in O_{\Phi}(\bar y)$ outside the class $O_{\Lambda}(\bar y)$.

We next claim that there exists a tuple $\bar a \in O_{\Lambda}(\bar y)$ such that all elements of the tuple $(\bar a, \bar y)$ are distinct. Otherwise in every tuple $(\bar a, \bar y)$ with $\bar a \in O_{\Lambda}(\bar y)$ an equation $a_i = y_j$ holds; we will derive a contradiction. For all $i\in\omega$, pick  $f_i\in\Phi$ such that for all $i\neq j$ the tuples $f_i(\bar y)$ and $f_j(\bar y)$ contain no common values. This is possible by the construction of $\cA^*$. By our assumption, for every function $g$ in the coset $f_i\Lambda$, $g(\bar y)$ contains an element of the tuple $f_i(\bar y)$. By the choice of the $f_i$, it follows that for $i\neq j$, the cosets $f_i\Lambda$ and $f_j\Lambda$ are disjoint. This is a contradiction to the finite index of $\Lambda$ in $\Phi$.

There exists $\bar d \in O_{\Phi}(\bar y)$ such that the tuples $(\bar y, \bar a), (\bar d, \bar a)$, and $(\bar d, \bar b)$ lie in the same orbit with respect to the action of $\Phi$. This follows from the extension property of $\cA^*$. %
So there are functions $h_1,h_2 \in \Phi$ such that $h_1(\bar y, \bar a) = (\bar d, \bar a)$ and $h_2(\bar d, \bar a) = (\bar d, \bar b)$. Since $\Phi$ preserves our partition, $h_1(\bar y, \bar a) = (\bar d, \bar a)$ implies that $\bar d$ lies in $O_{\Lambda}(\bar y)$. But because of $h_2(\bar d ,\bar a) = (\bar d, \bar b)$ also $\bar b$ lies in the same class, which is a contradiction.

We have shown the proposition for $\R = \prod_{n \geq 1} \Sym(n)$. Let now $\R'$ be an arbitrary separable profinite group. As such, it is topologically isomorphic to a closed subgroup of $\R$, so without loss of generality let $\R' \leq \R$. We set $\Sigma_{\R'}$ to be the preimage of $\R'$ under $\mu_{\R}$. Clearly then $\R' \cong_T \Sigma_{\R'} / \Phi$. Again $\Phi$ is the intersection of all the stabilizers of $X_n$ in $\Sigma_{\R'}$ for $n \geq 1$, implying that the intersection of all open subgroups of finite index in $\Sigma_{\R'}$ is contained in $\Phi$. Since $\Phi$ has no proper open subgroup of finite index, they are equal.
\end{proof}

\begin{nota}
From now on let $\Phi$ be the oligomorphic permutation group defined in the proof of Proposition~\ref{prop:encoding_1} and $A$ be its domain. Also let $\mu_\R\colon \Sigma_\R \to \R$ be the quotient mapping described in the proof.
\end{nota}

\subsection{Lifting the isomorphism to the encoding groups}\label{sect:lift1}

Let $\G$ be as in Proposition~\ref{prop:profinite_1}. The most natural next step in the proof might be to lift the non-topological isomorphism between $\G$ and $\G':=\G/\F\mult\F$ to an isomorphism between $\Sigma_\G$ and $\Sigma_{\G'}$. However, we do not know if this is possible. Instead, we will work with $\Sigma_{\G/\F}\times \F$ and the closure of $\Sigma_{\G/\F}$ in a discontinuous action as a permutation group.

As technical preparation for this, we will now provide a particular representation of the topological group $\G$ as a permutation group (i.e., a topological isomorphism with a permutation group).

\subsubsection{A representation of $\G$ as a permutation group}
As a separable profinite group, $\G$ contains a countable sequence $(\G_i)_{i\in\omega}$ of open normal subgroups with trivial intersection. Since $\G$ is compact, the factor groups $\G/\G_i$ are finite.
 Letting $\G$ act on the disjoint union of the factor groups by translation, we obtain a topologically faithful action of $\G$, i.e., a representation of $\G$ as a closed permutation group on the countable set $\bigcup_{i\in\omega}\G/\G_i$. In particular, we then have a representation of the subgroup $\F'$ as a (non-closed) permutation group on $\bigcup_{i \in \omega}\G/\G_i$.

Recall that $\F'$ is naturally isomorphic to $\G/\F$, but not topologically isomorphic to it. In the following, we will pick the open normal subgroups $\G_i$ mentioned above in such a way that the restriction of the action of $\F'$ to $\bigcup_{i\geq 1}\G/\G_i$ (where $\G/\G_0$ is missing) will still be faithful and hence isomorphic to $\F'$; however, it will be topologically isomorphic to $\G/\F$, and in particular not topologically isomorphic to $\F'$. Note that the topology on $\G/\F$ is obtained from the topology of $\F'$ by factorizing modulo $\F$, and hence is coarser than the topology on $\F'$, making such an undertaking possible.

Our action of $\G$ on $\bigcup_{i\in\omega}\G/\G_i$ will moreover have the property that its restriction to $\F'$ will be isomorphic (as an action) to an action of $\F'$ on the disjoint union $\bigcup_{i\in\omega}\F'/\F'_i$ of certain coset spaces of $\F'$, rather than of $\G$. Hence, it can be defined from $\F'$ alone. In particular, since $\F'$ is dense in $\G$, the action of $\G$ can be reconstructed from $\F'$ and a particular sequence of normal subgroups $(\F'_i)_{i\in\omega}$ thereof. Note that not all of the $\F'_i$ will be open, since the action of $\F'$ on $\bigcup_{i\in \omega}\F'/\F'_i$ is not a closed permutation group. In fact, only $\F'_0$ will be non-open.

It is this particular representation of $\G$ as a permutation group which will allow us to lift isomorphisms to the oligomorphic permutation groups encoding our profinite groups. Note that we use the particular structure of $\G$, e.g., the density of $\F'$, to obtain the representation.

To obtain the desired open normal subgroups, we first pick a sequence $(\H_i)_{i\geq 1}$ of open normal subgroups of $\G/\F$ whose intersection is the identity. The sequence exists since $\F$ is closed and so $\G/\F$ is profinite. We now set $\G_i$ to be the preimage of $\H_i$ under the quotient mapping, i.e., $\G_i:=\{hf\;|\; hF\in \H_i\text{ and } f\in\F\}$, for all $i\geq 1$. So each $\G_i$ is an open normal subgroup of $\G$, and $\bigcap_{i\geq 1}\G_i=\F$. To finish the construction, we pick an open normal subgroup $\G_0$ of $\G$ whose intersection with $\F$ is the identity; this is possible, because by profiniteness $\G$ contains a sequence of open normal subgroups with trivial intersection, and because $\F$ is finite. 
Finally, we set $\F'_i:=\F'\cap\G_i$, for all $i\in\omega$.

\begin{nota}
We now fix $(\G_i)_{i\in\omega}$ and $(\F'_i)_{i\in\omega}$ as above, and let $\tau\colon \G\To \Sym(\bigcup_{i\in\omega}\G/\G_i)$ be the mapping which sends an element $g$ of $\G$ to the permutation acting on $\bigcup_{i\in\omega}\G/\G_i$ by translation with $g$.
\end{nota}

\begin{lem}\label{lem:representation}\ 
\begin{enumerate}
\item[(1)] $\tau$ is faithful and continuous;
\item[(2)] $\F$ is the stabilizer of $\bigcup_{i \geq 1}\G/\G_i$ under the action $\tau$;
\item[(3)] the restriction of $\tau(\F')$ to $\bigcup_{i \geq 1}\G/\G_i$ is a permutation group that is topologically isomorphic to $\G/\F$;
\item[(4)] the actions of $\F'$ on $\bigcup_{i\in\omega}\G/\G_i$ (via $\tau$) and on $\bigcup_{i\in\omega}\F'/\F_i'$ (by translation) are isomorphic.
\item[(5)] the closure of $\F'$ in $\Sym(\bigcup_{i \in \omega} \F'/\F_i')$ is isomorphic to $\G$. 
\end{enumerate}
\end{lem}

\begin{proof} \
\begin{enumerate}
\item[(1)] The elements of the family $(\G_i)_{i\in\omega}$ are open normal subgroups of $\G$ with trivial intersection. Thus $\tau$ is faithful and continuous.
\item[(2)] Since $\F$ is the intersection of all $(\G_i)_{i \geq 1}$, it is the stabilizer of $\bigcup_{i \geq 1}\G/\G_i$.
\item[(3)] For every $i \geq 1$ the quotient group $\G/\G_i$ is isomorphic to $(\G/\F) / (\G_i/\F)$. Thus the action of $\F'$ on $\bigcup_{i\geq 1}\G/\G_i$ is isomorphic to the action of $\F'$ on $\bigcup_{i\geq 1}(\G/\F) / (\G_i/\F)$, which is a representation of $\G/\F$ as permutation group since the intersection of the factors $(\G_i/\F)$ is trivial by choice of the $\G_i$.
\item[(4)] Since $\F'$ is dense in $\G$ and all $\G_i$ are open, every coset in $\bigcup_{i\in\omega}\G/\G_i$ contains an element of $\F'$. Thus
\[ \G/\G_i = \F'\G_i/\G_i \cong \F'/(\F' \cap \G_i) = \F'/\F_i'. \]
One can now easily verify that the actions of $\F'$ on $\bigcup_{i\in\omega}\G/\G_i$ and on $\bigcup_{i\in\omega}\F'/\F_i'$ are isomorphic.
\item[(5)] This follows from (4) as $\F'$ is dense in $\G$. 
\end{enumerate}
\end{proof}

\subsubsection{The lifting}\label{sect:lift}
We will now consider a discontinuous action of $\Sigma_{\G/\F}$, similarly to the action of $\F'$ on $\bigcup_{i\in\omega}\G/\G_i$ in Lemma~\ref{lem:representation}, which is discontinuous if considered as an action of $\G/\F$ rather than of $\F'$: otherwise it would be closed as a permutation group, but its closure as a permutation group is topologically isomorphic to $\G$.

The quotient homomorphism $\mu_{\G/\F}\colon \Sigma_{\G/\F} \to \G/\F$ from Proposition~\ref{prop:encoding_1} gives rise to an action of $\Sigma_{\G/\F}$ on the cosets $\bigcup_{i \in \omega} \G/\G_i$ by simply considering the composition $\tau\circ\kappa\circ \mu_{\G/\F}$. If we restrict this action to $\bigcup_{i \geq 1} \G/\G_i$ then it is continuous, as the composition of continuous functions. But if we regard the action on $\bigcup_{i \in \omega} \G/\G_i$, the action fails to be continuous, since the induced permutation group is topologically isomorphic to the non-closed $\F'$.

Recall that $\Sigma_{\G/\F}$ was defined as a closed, oligomorphic permutation group on a countable set $A$. Clearly, the combined action of $\Sigma_{\G/\F}$ on $A \cup \G/\G_0$ fails to be continuous. By $\chi$ we denote the embedding of $\Sigma_{\G/\F}$ into $\Sym(A \cup \G/\G_0)$. Then, analogously to $\F'$ in the profinite case, $\chi[\Sigma_{\G/\F}]$ is not closed in $\Sym(A \cup \G/\G_0)$.

\begin{nota}
Henceforth $\chi$ will denote the action of $\Sigma_{\G/\F}$ on $A \cup \G/\G_0$, and $\Gamma$ the closure of $\chi[\Sigma_{\G/\F}]$ in $\Sym(A \cup \G/\G_0)$.
\end{nota}

Figure~\ref{fig:table_actions} gives an overview to all the group actions we are considering.

\begin{figure}[hbt]

\begin{tabular}{| lllll | ll |}
\hline
 & Group & acting on & via & properties & image $\cong_T$ & properties \\
\hline
(i) & $\G$ & $\bigcup_{i\in \omega}\G/\G_i$ & $\tau$ & faithful, cont. & $\G$ & closed\\
(ii) & $\G/\F$ & $\bigcup_{i\in \omega}\G/\G_i$ &$\tau\circ\kappa$& faithful,  discont. & $\F'$ & non-closed \\
(iii) & $\Sigma_{\G/\F}$ & $\bigcup_{i\in \omega}\G/\G_i$ &$\tau\circ\kappa\circ \mu_{\G/\F}$& discontinuous & $\F'$ &non-closed \\
(iv) & $\Sigma_{\G/\F}$ & $\bigcup_{i \geq 1}\G/\G_i$ &restr.~of (iii)& continuous & $\G/\F$ & closed \\
(v) & $\Sigma_{\G/\F}$ & $A \cup \G/\G_0$ &$\chi$& faithful, discont. & $\chi[\Sigma_{\G/\F}]$ &oligom., non-closed \\
(vi)  & $\Gamma$ & $A \cup \G/\G_0$ &ext.~of (v)& faithful, cont. & $\Gamma$ & oligom., closed \\
(vii) & $\Gamma$ & $\bigcup_{i\in \omega}\G/\G_i$  &comb.~of (iv), (vi)& continuous & $\G$ & closed\\
\hline
\end{tabular}
\caption{Group actions, some of their properties, and the permutation groups they induce.}
\label{fig:table_actions}
\end{figure}
 
\begin{lem}\label{lem:lifting_2} \
\begin{enumerate}
	\item[(1)] $\Gamma$ is a closed oligomorphic permutation group.
	\item[(2)] $\Gamma$ is the semidirect product $\chi[\Sigma_{\G/\F}] \cdot \Gamma_{(A)}$.
	\item[(3)] $\chi[\Phi]$ is the intersection of the open subgroups of finite index in $\Gamma$.
	\item[(4)] $\Gamma/\chi[\Phi] \cong_T \G$.
	\item[(5)] $\Gamma_{(A)}$ is central in $\Gamma$ and isomorphic to $\F$.
	\item[(6)] $\Gamma \cong \Sigma_{\G/\F} \times \F$.
\end{enumerate}
\end{lem}

\begin{proof} \ 
\begin{enumerate}
	\item[(1)] $\Gamma$ is closed by definition. As $\Sigma_{\G/\F}$ is oligomorphic on $A$ and $\G/\G_0$ is finite, it follows that $\Gamma$ is oligomorphic. 
	\item[(2)] The restriction function of $\Gamma$ to $A$ is a continuous homomorphism $|_A\colon \Gamma \to \Sigma_{\G/\F}$. Let $g\in \Gamma$ and let $(h_n)_{n \in \omega}$ be a sequence of permutations in $\Sigma_{\G/\F}$ such that $\chi(h_n)$ converges to $g$ in $\Gamma$. Then $(h_n)_{n \in \omega}$ converges in $\Sigma_{\G/\F}$, since $h_n= |_A\circ\chi(h_n) $ for all $n\in\omega$. By $h$ we denote its limit in $\Sigma_{\G/\F}$. The functions $g$ and $h$ are identical on $A$, thus $\chi(h)^{-1} \circ g \in \Gamma_{(A)}$. Moreover, $\chi[\Sigma_{\G/\F}]$ and $\Gamma_{(A)}$ have trivial intersection. Therefore $\Gamma$ is the semidirect product of $\chi[\Sigma_{\G/\F}]$ and $\Gamma_{(A)}$.
	\item[(3)] Note that $\chi$ is open and that it maps subgroups of finite index in $\Sigma_{\G/\F}$ to subgroups of finite index in $\Gamma$ by~(2). Since by Proposition~\ref{prop:encoding_1} the permutation group $\Phi$ is the intersection of the open subgroups of finite index in $\Sigma_{\G/\F}$, we have that $\chi[\Phi]$ contains the intersection of open subgroups of finite index in $\Gamma$. 
	
	For the other inclusion we remark that $\chi[\Phi]$ fixes $\G/\G_0$. Therefore the restriction of $\chi$ to $\Phi$ is continuous. If now $\chi[\Phi]$ had a proper open subgroup of finite index, then its preimage under $\chi$ would be open and of finite index in $\Phi$. Because of Proposition~\ref{prop:encoding_1} it would be equal to $\Phi$, a contradiction.
		
	\item[(4)] By considering $\mu_{\G/\F}\circ |_{A}$ we get a continuous surjective homomorphism of $\Gamma$ onto $\G/\F$. This gives us a continuous action of $\Gamma$ on $\bigcup_{i \geq 1} \G/\G_i$, by further composing with the mapping $\tau\circ \kappa$. By additionally letting $\Gamma$ act on $\G/\G_0$ by restriction of its domain we get a continuous action of $\Gamma$ on $\bigcup_{i \in \omega} \G/\G_i$ (Item~(vii) in Figure~\ref{fig:table_actions}). 
	
	It is easily verified that $\chi[\Phi]$ is the kernel of the action of $\Gamma$ on $\bigcup_{i \in \omega} \G /\G_i$. So $\Gamma / \chi[\Phi]$ is topologically isomorphic to the permutation group that $\Gamma$ induces on ${\bigcup_{i \in \omega} \G/\G_i}$ via this action. By the definition of the action, if we consider its restriction to 	$\chi[\Sigma_{\G/\F}]$, then it induces the same permutation group on $\bigcup_{i \in \omega} \G/\G_i$ as the action of $\G/\F$ on $\bigcup_{i \in \omega} \G/\G_i$ -- this permutation group is, by Lemma~\ref{lem:representation}, topologically isomorphic to $\F'$. Since the action of $\Gamma$ is continuous and $\Gamma$ is the topological closure of $\chi[\Sigma_{\G/\F}]$ we get that the permutation group it induces is topologically isomorphic to the closure of the action of $\F'$ on $\bigcup_{i \in \omega} \G /\G_i$, which is in turn topologically isomorphic to $\G$. In conclusion we get that $\Gamma/\chi[\Phi] \cong_T \G$.
	\item[(5)] In the action of $\Gamma$ on $\bigcup_{i \in \omega} \G /\G_i$ from~(4), the stabilizer of $\bigcup_{i \geq 1} \G /\G_i$ consists precisely of the elements of $\chi[\Phi] \cdot \Gamma_{(A)}$; this follows from~(2) and the definition of the action. Since the permutation group induced by this action on $\bigcup_{i \in \omega} \G /\G_i$ coincides with the permutation group induced by the action $\tau$ of $\G$ on this set, and since the stabilizer of $\bigcup_{i \geq 1} \G /\G_i$ in the latter action is isomorphic to $\F$, we get that $\chi[\Phi] \cdot \Gamma_{(A)}$, factored by the kernel $\chi[\Phi]$, is isomorphic to $\F$. Hence, $\Gamma_{(A)}$ is isomorphic to $\F$. As $\F$ is a central subgroup of $\G$, $\Gamma_{(A)}$ is central in $\Gamma$.
	\item[(6)] Since $\Gamma_{(A)}$ is a central normal subgroup, the semidirect product in (2) is a direct product. We conclude that, as groups:
	\[ \Gamma = \chi[\Sigma_{\G/\F}] \cdot \Gamma_{(A)} \cong \Sigma_{\G/\F} \mult \F.\]

\end{enumerate}
\end{proof}

\begin{nota} \label{310}
Let $\Delta$ be any closed oligomorphic permutation group on a countable set which is topologically isomorphic to $\Sigma_{\G/\F}\mult\F$. The existence of $\Delta$ follows from the fact that $\Sigma_{\G/\F}$ is itself such a group and that $\F$ is finite.
\end{nota}

\begin{cor}\label{cor:groups}
The closed oligomorphic permutation groups $\Delta$ and $\Gamma$ are isomorphic, but not topologically isomorphic.
\end{cor}

\begin{proof}
As we have seen in Lemma~\ref{lem:lifting_2}~(6), $\Delta$ and $\Gamma$ are isomorphic as groups. Recall that $\chi[\Phi]$ is the intersection of the  open subgroups of finite index in $\Gamma$, by Lemma~\ref{lem:lifting_2}~(3). By Proposition~\ref{prop:encoding_1}, $\Phi$ is the intersection of the open subgroups of finite index in $\Sigma_{\G/\F}$, and hence also in $\Sigma_{\G/\F}\mult \F$. Thus any topological isomorphism from $\Gamma$ to $\Delta$
 sends $\chi[\Phi]$ onto $\Phi$, and hence induces a topological isomorphism between the quotients $\Gamma / \chi[\Phi] \cong_T \G$ and $(\Sigma_{\G/\F} / \Phi) \times \F \cong_T \G/\F \times \F$, which is a contradiction to Lemma~\ref{lem:profinite_2}.
\end{proof}

\subsection{Extending the isomorphism to the closures of the groups}\label{sect:monoids}

\begin{nota}
For a permutation group $\Theta$, we denote by $\overline{\Theta}$ the topological closure of $\Theta$ in the space of all transformations on its domain, equipped with the topology of pointwise convergence.
\end{nota}
Note that the elements of $\overline{\Theta}$ are precisely the elementary embeddings to itself of any structure whose automorphism group is $\Theta$. Our aim in this section is to show that the monoids $\overline{\Delta}$ and $\overline{\Gamma}$ are isomorphic, but not topologically isomorphic. It is clear that $\overline{\Delta}$ and $\overline{\Gamma}$ are not topologically isomorphic, since the subgroups of invertible elements $\Delta$ and $\Gamma$ are not. It is harder to show that they are isomorphic, since there seems to be no obvious way to carry it over from the permutation groups, the problem being the non-continuity of the isomorphism. We therefore need to further study the topological monoids $\overline{\Delta}$ and $\overline{\Gamma}$ and how they are related to the profinite group $\G$.


\begin{lem}\label{lem:lifting_monoids1}
Let $\R$ be any separable profinite group. The continuous homomorphism $\mu_\R\colon \Sigma_\R \to \Sigma_\R/\Phi \cong_T \R$ extends to a continuous monoid homomorphism $\overline{\mu_\R}\colon \overline{\Sigma_\R} \to \R$.
\end{lem}

\begin{proof}
Recall that $\mu_\R$ was obtained via the action of $\Sigma_\R$ on $\bigcup_{n \geq 1} X_n$, where $X_n$ consists of the equivalence classes of the relations $E^n$. Every element of $\overline{\Sigma_\R}$ agrees on every finite set with an element of $\Sigma_\R$. Therefore the functions in $\overline{\Sigma_\R}$ preserve the equivalence relations $E^n$ and their negations for $n \geq 1$. Since every such relation has only finitely many equivalence classes, every element of  $\overline{\Sigma_\R}$ induces a permutation on them. This action of $\overline{\Sigma_\R}$ on $\bigcup_{i \geq 1} X_i$ extends the action of $\Sigma_\R$ and gives us the continuous monoid homomorphism $\overline{\mu_\R}$.
\end{proof}

Recall the discontinuous action of $\Sigma_{\G/\F}$ on the cosets $\bigcup_{i \in \omega} \G/\G_i$ via the mapping $\tau\circ\kappa\circ\mu_{\G/\F}$ (Item~(iii) in Figure~\ref{fig:table_actions}). With the help of $\overline{\mu_{\G/\F}}$ we see that this action has a natural extension to $\overline{\Sigma_{\G/\F}}$. As before, the restriction of this action to $\bigcup_{i \geq 1} \G/\G_i$ is continuous, and the induced permutation group is isomorphic to $\G/\F$. It is with the action on $\G/\G_0$ that we lose the continuity. 

By composing the continuous function $\overline{\mu_{\G/\F}} \circ |_A\colon \overline{\Gamma} \to \G/\F$ with the continuous action $\tau\circ\kappa$ of $\G/\F$ on $\bigcup_{i \geq 1} \G/\G_i$, we obtain a continuous action of $\overline{\Gamma}$ on $\bigcup_{i \geq 1} \G/\G_i$. By additionally letting $\overline{\Gamma}$ act on $\G/\G_0$ by restriction, we get a continuous action of $\overline{\Gamma}$ on $\bigcup_{i \in \omega}\G/\G_i$ which extends the action of $\Gamma$ thereon. 

Similarly to the situation with $\Sigma_{\G/\F}$, we can let $\overline{\Sigma_{\G/\F}}$ act on $A \cup \G/\G_0$, inducing an embedding $\overline{\chi}$ of $\overline{\Sigma_{\G/\F}}$ into the set of all transformations on $A \cup \G/\G_0$ which extends the group embedding $\chi$ from Lemma \ref{lem:lifting_2}.

\begin{lemma} \label{lem:lifting_monoids2} \
\begin{itemize}
	\item[(1)] $\overline{\Gamma}=\overline{\chi[\Sigma_{\G/\F}]}$.
	\item[(2)] All elements of  $\overline{\Gamma}$ which stabilize (pointwise) $A$ are invertible. Hence, $\overline{\Gamma}_{(A)} = \Gamma_{(A)}$.
	\item[(3)] The action of $\overline{\Gamma}$ on $\bigcup_{i \in \omega}\G/\G_i$ induces a permutation group that is equal  to $\G$.
	\item[(4)] $\overline{\Gamma}$ is isomorphic to the monoid direct product $\overline{\Sigma_{\G/\F}}\times \F$.
\end{itemize}
\end{lemma}

\begin{proof} \ 
\begin{itemize}
\item[(1)] $\Gamma$ was defined as the topological closure of $\chi[\Sigma_{\G/\F}]$ in $\Sym(A \cup \G/\G_0)$, so this is immediate. 

\item[(2)] The functions in $\overline{\Gamma}$ are injective, so by finiteness of $\G/\G_0$ any element of $\overline{\Gamma}$ which fixes all points of $A$ is bijective. 

\item[(3)] The action of $\Gamma$ on $\bigcup_{i \in \omega}\G/\G_i$ induces a permutation group that is topologically isomorphic to $\G$, by Lemma~\ref{lem:lifting_2}~(4). The action of $\overline{\Gamma}$ on $\bigcup_{i \in \omega}\G/\G_i$ extends this action. Since all permutations induced by the action of $\Gamma$ have only finite orbits, and since the action of $\overline{\Gamma}$ is continuous, every element of $\overline{\Gamma}$ actually induces a permutation on $\bigcup_{i \in \omega}\G/\G_i$. Every such permutation is in turn already induced by the action of $\Gamma$, since the permutation group induced by this action is closed. Summarizing, the functions induced by the two actions coincide, and induce a permutation group which is topologically isomorphic to $\G$.

\item[(4)] Let $g \in \overline{\Gamma}$, and assume first that $g$ fixes $\bigcup_{i \geq 1}\G/\G_i$ pointwise. So $g\rest_A \in \overline{\Sigma_{\G/\F}}$ fixes $\bigcup_{i \geq 1} \G/\G_i$ and $\overline{\chi}(g\rest_A)$ is the identity on $\G/\G_0$. Note that $\overline{\chi}(g\rest_A)$ agrees with $g$ on $A$ (but $g$ may be non-identity on $\G/\G_0$). 

By (3), there is $g' \in \Gamma$ which agrees with $g$ on $\bigcup_{i\in \omega} \G/\G_i$. By \ref{lem:lifting_2} we can write $g' = e\cdot f$ where $e \in \chi(\Sigma_{\G/\F})$ and $f \in \Gamma_{(A)}$. As $f, g'$ fix all of $\bigcup_{i\geq 1} \G/\G_i$, the same is true of $e$. So $e \in \chi(\Phi)$ and therefore $e$ fixes all of $\G/\G_0$. Thus $f \in \Gamma_{(A)}$ agrees with $g'$ and therefore with $g$ on $\bigcup_{i \in \omega} \G/\G_i$. So $g = \overline{\chi}(g\rest_A)\cdot f$ (as $g$ agrees with $\overline{\chi}(g\rest_A)$ on $A$ and $f$ fixes all of $A$). 

Now let $g$ be arbitrary. There exists $h\in \Sigma_{\G/\F}$ such that $g$ and $h$ agree in their action on $\bigcup_{i \geq 1}\G/\G_i$, by (3). Then by the preceding case, $\chi(h)^{-1}\circ g$ is contained in $\overline{\chi}[\overline{\Sigma_{\G/\F}}]\cdot \Gamma_{(A)}$, and hence so is $g$.

Clearly $\overline{\chi}[\overline{\Sigma_{\G/\F}}]\cap \Gamma_{(A)}$ is the trivial group and $\overline{\chi}$ is a monoid isomorphism from $\overline{\Sigma_{\G/\F}}$ to its image. As $\Gamma_{(A)} \cong \F$, we have the result.

%
%
\end{itemize}
\end{proof}

Let $\Delta$ be as in Notation \ref{310}. 

\begin{prop}\label{prop:monoids}
The closed transformation monoids $\overline{\Delta}$ and $\overline{\Gamma}$ are isomorphic, but not topologically isomorphic.
\end{prop}

\begin{proof}
The group $\Delta$ is topologically isomorphic to $\Sigma_{\G/\F} \times \F$, thus $\overline{\Delta}$ is topologically isomorphic to $\overline{\Sigma_{\G/\F}}\times\F$. By Lemma~\ref{lem:lifting_monoids2} $\overline{\Gamma}$ is isomorphic to $\overline{\Sigma_{\G/\F}}\times\F$, so $\overline{\Delta}$ and $\overline{\Gamma}$ are isomorphic. If they were topologically isomorphic, then also the  groups of invertible elements, equal to $\Delta$ and $\Gamma$ respectively, would be topologically isomorphic. But this contradicts Corollary~\ref{cor:groups}.
\end{proof}

\subsection{Extending the isomorphism to the function clones}

When $\Delta$ is any set of finitary functions on a given set, then there exists a smallest function clone containing it, the \emph{function clone generated by $\Delta$}. In the special case where $\Delta$ is a transformation monoid, this clone consists precisely of those functions which arise by adding dummy variables to the functions of the monoid. In this case, if $\Delta$ is topologically closed, then so is the function clone generated by $\Delta$. Thus moving from a $\Delta$ to the clone it generates is an algebraic procedure, in contrast to the moving from a closed permutation group to its topological closure as a transformation monoid, which is topological. It is therefore much more straightforward to extend non-topological isomorphisms between closed transformation monoids to the clones they generate. The following proposition is easy, its proof can be found in \cite{Reconstruction}.

\begin{prop}\label{prop:extending2clones}
Let $\Sigma, \Lambda$ be transformation monoids, and let $\xi\colon \Sigma\To\Lambda$ be a monoid isomorphism such that both $\xi$ and its inverse function send constant functions to constant functions. Then $\xi$ extends to an isomorphism between  the function clones generated by $\Sigma$ and $\Lambda$.
\end{prop}

\begin{cor}\label{cor:extending2clones}
The function clones generated by the transformation monoids $\overline{\Delta}$ and $\overline{\Gamma}$ are isomorphic, but not topologically isomorphic.
\end{cor}
\begin{proof}
By Propositions~\ref{prop:monoids} and~\ref{prop:extending2clones}, the clones are isomorphic. Any topological isomorphism between them would yield a topological isomorphism between the monoids $\overline{\Delta}$ and $\overline{\Gamma}$ by restriction to the unary sort, and hence contradict Proposition~\ref{prop:monoids}.
\end{proof}

\subsection{Encoding into a finite relational language} 

We have shown that there are $\omega$-categorical structures $\mathcal A$ and $\mathcal B$ whose
endomorphism monoids are isomorphic, but not topologically isomorphic. The structure $\mathcal A$ has an infinite
signature, and it is easy to see from the theorem of Coquand, Ahlbrandt, and Ziegler~\cite{AhlbrandtZiegler} that any structure $\mathcal A'$ whose automorphism group is topologically isomorphic to the one of $\mathcal A$ must have an infinite signature.  In this section we are going to show that there is an $\omega$-categorical structure in a finite language such that its automorphism group, its endomorphism monoid and its polymorphism clone do not have reconstruction. 

The key ingredient for the counterexamples of the previous sections was Proposition \ref{prop:encoding_1}. It gave us an encoding of the profinite group $\G/\F$ as the quotient of an oligomorphic group $\Sigma_{\G/\F}$ and the intersection of its open subgroups of finite index. Our primary goal in this section is to construct an oligomorphic permutation group $\tilde \Sigma$ that also encodes $\G/\F$ in the above sense and can be written as the automorphism group of a structure with finite signature. 
We will obtain $\tilde \Sigma$ with the help of a theorem due to Hrushovski, which states that every $\omega$-categorical structure is definable on a definable subset of an $\omega$-categorical structure with finite signature. In Proposition \ref{prop:finite_1} we present Hrushovski's result and a proof sketch taken from \cite[Theorem 7.4.8]{Hodges} in order to refer to this construction later on.

\begin{prop} \label{prop:finite_1}
Let $\cA$ be a countable $\omega$-categorical structure. Then there is a finite language $L$, containing a 1-ary predicate $P$, and an $\omega$-categorical $L$-structure $\cB$, such that the domain of $\cA$ is equal to the elements of $\cB$ satisfying $P$ and the definable relations of $\cA$ are exactly the definable relations of $\cB$ restricted to $P$.
\end{prop}

\begin{proof} We can assume that $\cA$ is relational with atomic relations $R_1, R_2, \ldots$ where $R_n$ has arity $l(n)$.
We can also assume that every definable relation in $\Ac$ is equivalent to an atomic formula and that $l(n) \leq n$ for all $n\geq 1$. In particular,  $\Ac$ has quantifier elimination and is homogeneous. Let $L$ be the language consisting of the relation symbols $P$, $Q$, $\lambda$ and $\rho$ (all 1-ary), $H$ (2-ary), and $S$ (4-ary), and let $L^+$ be the union of $L$ and the language of $\cA$. Let $T$ be the theory in $L^+$ which says:\begin{itemize}
\item If $R_n(\bar x)$ for some $n\geq 1$, then all entries of $\bar x$ satisfy $P$;
\item $Q(x)$ if and only if $\neg P(x)$;
\item if $\lambda(x)$ or $\rho(x)$, then $Q(x)$;
\item if $H(x,y)$, then $Q(x)$ and $Q(y)$;
\item if $S(x,y,a,b)$ then $Q(x)$, $Q(y)$, and $P(a)$, $P(b)$.
\end{itemize}

Let $\cM$ be a model of $T$. Then we say a set of elements of $\cM$ is an \textit{$n$-pair} if it can be written as 
 $\{a_1,\ldots, a_{l(n)}, c_1, \ldots, c_n\}$, where $n\geq 1$ and
\begin{itemize}
\item $P(a_i)$ holds for all $1 \leq i \leq l(n)$ and $Q(c_i)$ holds for all $1 \leq i \leq n$;
\item the elements $c_i$ are distinct and $H(c_i,c_j)$ holds iff $j \equiv i+1 \mod n$;
\item $\lambda(c_i)$ holds iff $i = 1$ and $\rho(c_i)$ holds iff $i = l(n)$;
\item $S(c_h,c_i,a_k,a_m)$ holds iff $a_h = a_k$.
\end{itemize}
Note that if an $n$-pair  $\{a_1,\ldots, a_{l(n)}, c_1, \ldots, c_n\}$ is given, we can uniquely recover the sequence $(c_1,\ldots,c_n)$ and also the sequence $(a_1,\ldots,a_{l(n)})$, which may contain repetitions. We say the $n$-pair \textit{labels} the sequence $\bar a = (a_1,\ldots,a_{l(n)})$.

Consider the class of finite models $\cB'$ of $T$ such that 
\begin{itemize}
\item the restriction of $\cB'$ to $P$ and to the relations $R_n$ is isomorphic to a finite substructure of $\cA$;
\item for every $n\geq 1$, if $\cB'$ contains an $n$-pair which labels the sequence $\bar a$, then $B \models R_n(\bar a)$.
\end{itemize}
By~\cite[Theorem 7.4.8]{Hodges}, this is a Fra\"iss\'e class; let $\cB^+$ be its Fra\"iss\'e limit.


Clearly, the restriction of $\cB^+$ to the subset $P$ and to the relations $R_n$ is homogeneous and has the same age as $\Ac$. Therefore it is isomorphic to $\Ac$. Let $\cB$ be the reduct of $\cB^+$ in the language $L$. By construction $\cB^+\models R_n(\bar a)$ holds if and only if some $n$-pair in $\cB^+$ labels $\bar a$. Therefore every relation $R_n$ is definable in $\cB$.
\end{proof}

In Proposition \ref{prop:finite_1}, the definable relations of $\cB$ restricted to $P$ are exactly the definable relations of $\cA$. Hence the orbits of $\Aut(\cB)$ and the orbits of $\Aut(\cA)$ on tuples in $P$ coincide. However we do not know if the restriction of $\Aut(\cB)$ to $P$ is closed in the full group $\Aut(\cA)$, i.e. it might be a proper dense subgroup of $\Aut(\cA)$. 

\begin{lem}
Let $\cA$ be a countable $\omega$-categorical homogeneous structure and $\cB$ as constructed in Proposition \ref{prop:finite_1}. Then $\End(\cB) = \overline{\Aut(\cB)}$, i.e., $\cB$ is a \emph{model-complete core} (cf. \cite{manuel-core}).
\end{lem}

\begin{proof}
It is shown in \cite{RandomMinOps} that for $\omega$-categorical structure $\cB$, $\End(\cB) = \overline{\Aut(\cB)}$ holds if and only if every formula in $\cB$ is equivalent to an existential positive formula. 
Let $\cB^+$ be as in the proof of Proposition \ref{prop:finite_1}. Because of the homogeneity of $\cB^+$, every $L$-formula in $\cB$ is equivalent to a quantifier-free $L^+$-formula in $\cB^+$. So it suffices to show that every quantifier free $L^+$-formula is equivalent to an existential positive $L$-formula in $\cB^+$. We first prove the statement for an atomic formula $R_n(x_1,\ldots,x_{l(n)})$. By the construction of $\cB^+$ we have
\[ \cB^+ \models R_n(x_1,\ldots,x_{l(n)}) \Leftrightarrow \cB^+ \models \exists y_1, \ldots, y_n  \ ( \{ x_1,\ldots,x_{l(n)},  y_1, \ldots, y_n \} \text{ is an $n$-pair})\; . \]
The latter is an existential positive $L$-formula, since the definition of an $n$-pair did not require quantifiers or negations. For a general quantifier-free formula in $\cB^+$ we can assume that the relations $(R_n)_{n \geq 1}$ only appear in positive form, since we introduced a relation symbol for every definable relation in $\cA$. Applying the equivalence above for every such $R_n$ then gives us an existential positive formula in $L$.
\end{proof}

From now on, let $\cA$ be the canonical structure of the oligomorphic permutation group $\Sigma_{\G/\F}$, i.e., the structure on the domain of $\Sigma_{\G/\F}$ containing all relations which are invariant under $\Sigma_{\G/\F}$. Let $\cB$ and $\cB^+$ be as in the proof of Proposition \ref{prop:finite_1}. Set $\tilde\Sigma := \Aut(\cB)$, and let $\tilde\mu\colon \tilde\Sigma \to \G/\F$ be the composition of the restriction of $\tilde\Sigma$ to $P$ and the homomorphism $\mu_{\G/\F}$.


Recall the construction of $\Sigma_{\G/\F}$ in Proposition \ref{prop:encoding_1}. 
Let $\cA^*$ be, as in the proof of that proposition, the structure $(A,(P_i^n)_{1\leq i\leq n})$. Recall that $\cA^*$ is $\omega$-categorical and homogeneous, and that all relations of $\cA$ are definable in $\cA^*$. By $\cB^*$ we denote the expansion of $\cB^+$ with the relations $(P_i^n)_{1\leq i\leq n}$ on its $P$-part. Let $\tilde\Phi$ be the automorphism group of $\cB^*$.


\begin{lem} \label{lem:finite_2} The map $\tilde\mu\colon \tilde\Sigma \to \G/\F$ is a continuous surjective homomorphism with kernel $\tilde\Phi$. Furthermore, $\tilde\Phi$ is the intersection of the open subgroups of finite index in $\tilde\Sigma$.
\end{lem}

\begin{proof} As a composition of continuous homomorphisms, $\tilde\mu$ is a continuous homomorphism. As in Proposition \ref{prop:encoding_1} we can think about $\tilde\mu$ as an action of the elements of $\tilde \Sigma$ on the set $\bigcup_{n \geq 1} X_n$, where $X_n = \{P_1^n, \ldots, P_n^n \}$ for all $n\geq 1$. The functions in $\tilde \Phi$ are exactly those elements who stabilize all $P_i^n$ pointwise, so $\tilde \Phi$ is indeed the kernel of $\tilde\mu$. Using the homogeneity of $\cB^+$ and a back-and-forth argument as in Proposition~\ref{prop:encoding_1} one can show that $\tilde\mu$ is surjective. 

Note that the age of $\cB^*$ consists exactly of those structures whose restriction to $P$ lies in the age of $\cA^*$ and whose reduct to the language $L^+$ lies in the age of $\cB^+$. With this in mind it is easy to verify that $\cB^*$ satisfies the extension property. Hence also $\cB^*$ is homogeneous.

The subgroup of $\tilde\Sigma$ consisting of the elements that stabilize $X_n$ pointwise for a fixed $n\geq 1$ is open and of finite index. The intersection of all such subgroups is equal to $\tilde\Phi$. Hence the intersection of all open subgroups of finite index in $\tilde\Sigma$ is contained in $\tilde\Phi$.

It remains to show that also the other inclusion holds; we follow the proof of Proposition~\ref{prop:encoding_1}. Assume that $\tilde\Phi$ has a proper open subgroup $\tilde\Lambda$ of finite index. Because of the openness of $\tilde\Lambda$, there is a tuple $\bar y$ such that the stabilizer $\tilde\Phi_{(\bar y)}$ lies in $\tilde\Lambda$. Let $O_{\tilde\Phi}(\bar y)$ and $O_{\tilde\Lambda}(\bar y)$ denote the orbits of $\bar y$ under $\tilde \Phi$ and $\tilde \Lambda$, respectively. We will obtain a contradiction by studying the action of $\tilde\Phi$ on the partition of $O_{\tilde\Phi}(\bar y)$ into blocks $gO_{\tilde\Lambda}(\bar y)$ with $g \in \tilde \Phi$. The index $|\tilde\Phi : \tilde\Lambda|$ coincides with the number of partition classes $gO_{\tilde\Lambda}(\bar y)$ in $O_{\tilde\Phi}(\bar y)$.

Choose a tuple $\bar a \in O_{\tilde\Lambda}(\bar y)$ and a tuple $\bar b$ from another partition class such that the entries of $(\bar y, \bar a, \bar b)$ are pairwise disjoint. We claim that there is a $\bar d \in O_{\tilde\Phi}(\bar y)$ such that $(\bar y, \bar a)$, $(\bar d, \bar a)$ and $(\bar d, \bar b)$ lie in the same orbit of $\tilde\Phi$, which is a contradiction.

By the homogeneity of $\cB^*$ two tuples lie in the same orbit of $\tilde\Phi$ if they satisfy the same relations in $\cB^*$. We write $\bar y = (\bar y_P, \bar y_{\neg P})$, where the components of $\bar y_P$ satisfy $P$, and the components of $\bar y_{\neg P}$ do not satisfy $P$. Similarly, we write $\bar a = (\bar a_P, \bar a_{\neg P})$, $\bar b = (\bar b_P, \bar b_{\neg P})$. By the proof of Proposition \ref{prop:encoding_1}, we can find a tuple $\bar d_P$ of elements of  $\cA^*$ such that $(\bar y_P, \bar a_P)$, $(\bar d_P, \bar b_P)$ and $(\bar d_P, \bar a_P)$ satisfy the same relations. 

We wish to find a tuple $\bar d_{\neg P}$ of the same length as $\bar y_{\neg P}$ such that setting $\bar d:= (\bar d_P, \bar d_{\neg P})$ we have that $(\bar d, \bar a)$ and $(\bar d, \bar b)$ lie in the same orbit as $(\bar y, \bar a)$. To this end, let $\bar d_{\neg P}$ be a tuple of new variables of the right length. We endow the set of elements appearing in $\bar y, \bar a, \bar b$ and $\bar d$ with relations $\rho$, $\lambda$, $H$ and $S$ such that we obtain a structure in the age of $\cB^*$, and such that $(\bar d, \bar a)$ and $(\bar d, \bar b)$ satisfy the same relations as $(\bar y, \bar a)$; clearly, we can then realize these variables as elements of $\cB^*$ and are done by homogeneity. When doing so we can also ensure that all quadruples of elements from $\bar a$, $\bar b$, and $\bar d$ for which $S$ holds consist entirely of elements of $(\bar d, \bar a)$ or of $(\bar d, \bar b)$.

We claim that the resulting structure lies in the age of $\cB^*$. Assume otherwise. Then the reduct of the structure in $L^+$ contains a $n$-pair that labels a tuple $\bar x$ with $\neg R_n(\bar x)$. This $n$-pair has to contain elements of $\bar d$, otherwise this would be a contradiction to the fact that the union of the elements of $\bar y$, $\bar a$ and $\bar b$ induces a structure in the age of $\cB^*$. Moreover, this $n$-pair lies entirely in $(\bar d, \bar a)$ or $(\bar d, \bar b)$, since $S$ does not hold for any other tuples containing elements of $\bar d$. But then, by construction, also the union of $\bar y$ and $\bar a$ contains a $n$-pair that labels an $\bar x'$ with $\neg R_n(\bar x')$. This contradicts the fact that the union of the elements of $\bar a$ and $\bar y$ lies in the age of $\cB^*$. This proves our claim.

Therefore there are functions $h_1,h_2 \in \tilde\Phi$ such that $h_1(\bar y, \bar a) = (\bar d, \bar a)$ and $h_2(\bar d, \bar a) = (\bar d, \bar b)$. Since $\Phi$ preserves our partition, $h_1(\bar y, \bar a) = (\bar d, \bar a)$ implies that $\bar d$ lies in $O_{\Lambda}(\bar y)$. But because of $h_2(\bar d ,\bar a) = (\bar d, \bar b)$ also $\bar b$ lies in the very same class, which is a contradiction.
\end{proof}

We are now ready to conclude this section with the proof of Theorem~\ref{thm:finitelanguage}.

\begin{proof}[Proof of Theorem~\ref{thm:finitelanguage}]
In Lemma \ref{lem:finite_2} we have shown that $\tilde\mu: \tilde\Sigma \to \G/\F$ is a surjective continuous homomorphism whose kernel $\tilde\Phi$ is the intersection of open subgroups with finite index in $\tilde\Sigma = \Aut(\cB)$. Let $B$ be the domain of $\cB$. We proceed as in Section~\ref{sect:lift1}: Via $\tilde\mu$ we can define an action of $\tilde\Sigma$ on $B \cup \G/\G_0$. This action is not continuous and has a non-open image, let $\tilde\Gamma$ be its closure in $\Sym(B \cup \G/\G_0)$. Then, following the exact same proof steps as in Lemma~\ref{lem:lifting_2} and Corollary \ref{cor:groups} we see that $\tilde\Gamma$ and $\tilde\Sigma \times \F$ are isomorphic, but not topologically isomorphic. By the same arguments as in Section~\ref{sect:monoids} one can also prove that $\overline{\tilde\Gamma}$ and $\overline{\tilde\Sigma} \times \F$ are isomorphic as abstract monoids, but not topologically isomorphic.

Since $\F$ is finite and $\cB$ has finite signature, there is a structure $\cC$ with finite signature such that $\End(\cC) \cong_T \End(\cB) \times \F$. Then $\Aut(\cC)$ is topologically isomorphic to $\Aut(\cB) \times \F = \tilde\Sigma \times \F$, which we know does not have reconstruction. By the model completeness of $\cB$, we know that its automorphism group is dense in its endomorphism monoid. It follows that $\End(\cC) \cong_T \End(\cB) \times \F = \overline{\tilde\Sigma} \times \F$,
proving that also the endomorphism monoid of $\cC$ has no reconstruction. Finally, by including the relation $R(x,y,a,b) \leftrightarrow x=y \lor a=b$ in $\cC$ one can ensure that the polymorphism clone of $\cC$ consists of those functions arising from endomorphisms of $\cC$ by adding dummy variables. By Proposition \ref{prop:extending2clones}, $\Pol(\cC)$ and the function clone generated by $\tilde\Gamma$ are isomorphic, but not topologically isomorphic.
\end{proof}

We do not know whether $\tilde\Gamma$  can be represented as automorphism group of a structure with finite relational signature. Similarly, we do not know whether its closure $\overline{\tilde\Gamma}$ as a monoid is the endomorphism monoid of a structure with finite relational signature. 

\section{Open Problems}
\label{sect:open}

Because of the comments on the consistency of
reconstruction for groups in Section~\ref{sect:profinite}, the following question is of central importance for the
reconstruction of structures from their endomorphism monoid. 

\begin{quest}
Let $\Sigma$ be a closed oligomorphic subgroup of $\Sym(\omega)$ which has reconstruction. Does the monoid obtained as the closure of $\Sigma$ in $\omega^\omega$ have reconstruction?
\end{quest}

A positive answer would imply that it is consistent 
with ZF+DC that all monoids with a dense set of units have reconstruction. These monoids play a central role
in the study of polymorphism clones of $\omega$-categorical structures, in particular for the study of the computational
complexity of constraint satisfaction problems (we refer to~\cite{Topo-Birk, BartoPinskerDichotomy} for details).

In the course of the proof, we encountered natural questions that we had to leave open (for example at the beginning of Section~\ref{sect:lift1}). An answer to the following question will
most probably shed some light on them. 

\begin{quest}
Let $\Gamma$ be a closed oligomorphic permutation group without reconstruction. Does the monoid closure of
$\Gamma$ also fail to have reconstruction?
\end{quest}

Lascar showed in~\cite{Lascar-demigroupe} that if $\mathcal A$ and $\mathcal B$ are countable $\omega$-categorical structures which are \emph{$G$-finite}, then any isomorphism between their endomorphism monoids is a topological isomorphism when restricted to their automorphism groups. An early version of that article concluded with the question whether the assumption of $G$-finiteness could be dropped; the published version does not contain the question anymore. We remark that our example would be a counterexample to that question.

\bibliographystyle{alpha}
\bibliography{cloneiso.bib}

\begin{thebibliography}{BPP14}

\bibitem[AZ86]{AhlbrandtZiegler}
Gisela Ahlbrandt and Martin Ziegler.
\newblock Quasi-finitely axiomatizable totally categorical theories.
\newblock {\em Annals of Pure and Applied Logic}, 30(1):63--82, 1986.

\bibitem[BK96]{BeckerKechris}
Howard Becker and Alexander Kechris.
\newblock {\em The Descriptive Set Theory of {P}olish Group Actions}.
\newblock Number 232 in LMS Lecture Note Series. Cambridge University Press,
  1996.

\bibitem[Bod05]{manuel-core}
Manuel Bodirsky.
\newblock The core of a countably categorical structure.
\newblock In {\em STACS 2005}, pages 110--120. Springer, 2005.

\bibitem[BP14]{RandomMinOps}
Manuel Bodirsky and Michael Pinsker.
\newblock Minimal functions on the random graph.
\newblock {\em Israel Journal of Mathematics}, 200(1):251--296, 2014.

\bibitem[BP15]{Topo-Birk}
Manuel Bodirsky and Michael Pinsker.
\newblock Topological {B}irkhoff.
\newblock {\em Transactions of the American Mathematical Society},
  367:2527--2549, 2015.

\bibitem[BP16]{BartoPinskerDichotomy}
Libor Barto and Michael Pinsker.
\newblock The algebraic dichotomy conjecture for infinite domain constraint
  satisfaction problems.
\newblock In {\em Proceedings of LICS'16}, 2016.
\newblock Preprint arXiv:1602.04353.

\bibitem[BPP]{Reconstruction}
Manuel Bodirsky, Michael Pinsker, and Andr\'{a}s Pongr\'acz.
\newblock Reconstructing the topology of clones.
\newblock {\em Transactions of the American Mathematical Society}.
\newblock To appear. Preprint available from arXiv:1312.7699.

\bibitem[BPP14]{BPP-projective-homomorphisms}
Manuel Bodirsky, Michael Pinsker, and Andr\'{a}s Pongr\'acz.
\newblock Projective clone homomorphisms.
\newblock Preprint arXiv:1409.4601, 2014.

\bibitem[EH90]{EvansHewitt}
David~M. Evans and Paul~R. Hewitt.
\newblock Counterexamples to a conjecture on relative categoricity.
\newblock {\em Annals of Pure and Applied Logic}, 46(2):201--209, 1990.

\bibitem[Hod97]{Hodges}
Wilfrid Hodges.
\newblock {\em A shorter model theory}.
\newblock Cambridge University Press, Cambridge, 1997.

\bibitem[Kec95]{Kechris}
Alexander Kechris.
\newblock {\em Classical descriptive set theory}, volume 156 of {\em Graduate
  Texts in Mathematics}.
\newblock Springer, 1995.

\bibitem[Las82]{Lascar82}
Daniel Lascar.
\newblock On the category of models of a complete theory.
\newblock {\em Journal Symbolic Logic}, 47(2):249--266, 1982.

\bibitem[Las89]{Lascar-demigroupe}
Daniel Lascar.
\newblock Le demi-groupe des endomorphismes d'une structure
  $\aleph_0$-cat\'{e}gorique.
\newblock In M.~Giraudet, editor, {\em Actes de la Journ\'{e}e Alg\`{e}bre
  Ordonn\'{e}e (Le Mans, 1987)}, pages 33--43, 1989.

\bibitem[Las91]{Lascar91}
Daniel Lascar.
\newblock Autour de la propri\'et\'e du petit indice.
\newblock {\em Proceedings of the London Mathematical Society}, 62(1):25--53,
  1991.

\bibitem[RZ00]{RibesZalesskii}
Luis Ribes and Pavel Zalesskii.
\newblock {\em Profinite Groups}.
\newblock Springer, 2000.

\bibitem[She84]{Shelah84}
Saharon Shelah.
\newblock Can you take {Solovay}'s inaccessible away?
\newblock {\em Israel Journal of Mathematics}, 48(1):1--47, 1984.

\bibitem[Wit54]{Witt}
Ernst Witt.
\newblock \"{U}ber die {K}ommutatorgruppe kompakter {G}ruppen.
\newblock {\em Rendiconti di Matematica e delle sue Applicazioni. Serie V.},
  14:125--129, 1954.

\end{thebibliography}

\end{document}